\documentclass[12pt,reqno]{amsart}
\usepackage{amsthm, amssymb}
\numberwithin{equation}{section}

\usepackage{bm}

{\theoremstyle{plain}%
  \newtheorem{theorem}{Theorem}[section]
  \newtheorem{corollary}[theorem]{Corollary}
  \newtheorem{proposition}[theorem]{Proposition}
  \newtheorem{lemma}[theorem]{Lemma}
  
  \newtheorem{example}[theorem]{Example}%
}
\newtheorem*{mThm}{Main Theorem}
\def\Sym{{\rm Sym\,\,}}%
\def\QSym{{\rm QSym\,\,}}%
\begin{document}
%
\title{Generalized Harmonic Number Sums and Quasi-Symmetric Functions}
\author[K.-W. Chen]{Kwang-Wu Chen}
\address[Kwang-Wu Chen]{Department of Mathematics\\ University of Taipei\\
No. $1$,  Ai-Guo West Road, Taipei $10048$, Taiwan}
\email[Corresponding author]{kwchen@uTaipei.edu.tw}
\today
\maketitle
\begin{abstract}  
We express some general type of infinite series such as
$$
\sum^\infty_{n=1}\frac{F(H_n^{(m)}(z),H_n^{(2m)}(z),\ldots,H_n^{(\ell m)}(z))}
{(n+z)^{s_1}(n+1+z)^{s_2}\cdots (n+k-1+z)^{s_k}},
$$
where $F(x_1,\ldots,x_\ell)\in\mathbb Q[x_1,\ldots,x_\ell]$,
$H_n^{(m)}(z)=\sum^n_{j=1}1/(j+z)^m$, $z\in (-1,0]$,
and $s_1,\ldots,s_k$ are nonnegative integers with $s_1+\cdots+s_k\geq 2$,
as a linear combination of multiple Hurwitz zeta functions 
and some speical values of $H_n^{(m)}(z)$.
\end{abstract}
\noindent{\small {\it Key Words:} 
Symmetric functions, Quasi-symmetric functions, Riemann zeta values, 
Multiple zeta values, Multiple Hurwitz zeta functions.}
\section{Introduction}
For $m,n\in\mathbb N$, $z\in\mathbb C$ with $z\neq -1,-2,\ldots,-n$,
the generalized harmonic functions are defined by
$H_n^{(m)}(z)=\sum^n_{k=1}\frac1{(k+z)^m}$. The generalized harmonic function 
$H_n^{(m)}(z)$ is a generalization of the generalized harmonic number $H_n^{(m)}(0)=H_n^{(m)}$.

In this paper we investigate series of the following form
\begin{equation}\label{eq.11}
\sum^\infty_{n=1}\frac{F(H_n^{(m)}(z),H_n^{(2m)}(z),\ldots,H_n^{(\ell m)}(z))}
{(n+z)^{s_1}(n+1+z)^{s_2}\cdots (n+k-1+z)^{s_k}},
\end{equation}
where $F(x_1,\ldots,x_\ell)\in\mathbb Q[x_1,\ldots,x_\ell]$,
$z\in (-1,0]$,
and $s_1,\ldots,s_k$ are nonnegative integers with $s_1+\cdots+s_k\geq 2$.

The famous such formulas:
$$
\sum^\infty_{n=1}\frac{H_n}{(n+1)^2}=\zeta(3),\quad
\mbox{and}\quad
\sum^\infty_{n=1}\frac{H_n}{n^3}=\frac54\zeta(4),
$$
were discovered by Euler in relation to Euler sums,
where the Riemann zeta function is defined by 
$\zeta(s)=\sum^\infty_{n=1}\frac1{n^s}$, $\Re(s)>1$.
Many similar formulas have been established since,
and there is an extensive mathematical literature; see, e.g. \cite{BB, Chu, Doe, FS, Zhe}.

Let $z$ be a real number in $(-1,0]$. The multiple Hurwitz zeta function 
of depth $k$ and weight $s_1+\cdots+s_k$ is defined by
\cite{AI, KM, MV}
$$
\zeta(s_1,\ldots,s_k;z)=\sum_{1\leq n_1<n_2<\cdots<n_k}
\frac{1}{(n_1+z)^{s_1}(n_2+z)^{s_2}\cdots(n_k+z)^{s_k}},
$$
which is absolutely convergent and analytic in the region 
Re$\,(s_\ell+s_{\ell+1}+\cdots+s_k)>k-\ell+1$, for $\ell=1,2,\ldots,k$.
It is well-known that $\zeta(s_1,\ldots,s_k;z)$ has a analytic continuation to $\mathbb C^k$.
Fix a positive integer $n$, the finite multiple Hurwitz zeta function is defined by
$$
\zeta_n(s_1,\ldots,s_k;z)=\sum_{1\leq n_1<n_2<\cdots<n_k\leq n}
\frac{1}{(n_1+z)^{s_1}(n_2+z)^{s_2}\cdots(n_k+z)^{s_k}}.
$$
It is known that multiple zeta values (MZVs) $\zeta(s_1,\ldots,s_k)=\zeta(s_1,\ldots,s_k;0)$ 
are introduced by Hoffman \cite{Hof3} and Zagier \cite{Zag},
and multiple $t$-values (MtVs) $t(s_1,\ldots,s_k)=2^{-(s_1+\cdots+s_k)}\zeta(s_1,\ldots,s_k;-1/2)$
are introduced by Hoffman \cite{Hof4}.

Recently a lot of formulas concerning $O_n^{(m)}=2^{-m}H_n^{(m)}(-1/2)=\sum^n_{k=1}\frac1{(2k-1)^m}$ 
were produced \cite{Chu2, Wang, WC}, for example,
\begin{eqnarray*}
\sum^\infty_{n=1}\frac{O_n^{(2)}}{(2n-1)(2n+1)} &\displaystyle= \frac7{16}\zeta(3)
&=\frac{t(3)}2,\quad\mbox{(\cite[Corollay 3]{Wang})} \\
\sum^\infty_{n=1}\frac{O_n^2}{(2n+1)(2n+3)} &=\displaystyle \frac18+\frac{\pi^2}{32}
&=\frac18+\frac{t(2)}4.\quad\mbox{(\cite[Corollary 4]{WC})}
\end{eqnarray*}
These may be compared with the following identities
\begin{eqnarray*}
\sum^\infty_{n=1}\frac{H_n^{(2)}}{n(n+1)} &=&\zeta(3),\quad\mbox{(\cite[Lemma 1.2]{Sof2})} \\
\sum^\infty_{n=1}\frac{H_n^2}{(n+1)(n+2)} &=&1+\zeta(2).\quad\mbox{(\cite[Corollary 3]{WC})}
\end{eqnarray*}
It appears that all these identities are special cases with $z=0$ and $z=-1/2$ of the following identities:
$$
\sum^\infty_{n=1}\frac{H_n^{(2)}(z)}{(n+z)(n+1+z)}=\zeta(3;z),\quad\mbox{and}\quad
\sum^\infty_{n=1}\frac{H_n(z)^2}{(n+1+z)(n+2+z)}=\frac1{1+z}+\zeta(2;z).
$$

In 2016, Hoffman \cite{Hof1} introduced a class of functions $\eta_{s_1,\ldots,s_k}$
from the quasi-symmetric function $\QSym$ to the reals in deal with the infinite series
$$
\sum^\infty_{n=1}\frac{F(H_n,H_n^{(2)},\ldots,H_n^{(\ell)})}
{n^{s_1}(n+1)^{s_2}\cdots (n+k-1)^{s_k}},
$$
He gave explicitly formulas
for these functions of length $k\leq 3$ and two specific functions 
$\eta_{s_1,\ldots,s_k}(e_\ell)$,
where $(s_1,\ldots,s_k)=(0,\underbrace{1,\ldots,1}_{k-1})$,
or $(\underbrace{1,\ldots,1}_k)$, and
$e_\ell$ is the $\ell$-th elementary symmetric polynomials. 
For example \cite[Theorem 2]{Hof1},
$$
\sum^\infty_{n=1}\frac{P_\ell(H_n,H_n^{(2)},\ldots,H_n^{(\ell)})
P_k(H_n,-H_n^{(2)},\ldots,(-1)^{k+1}H_n^{(k)})}{n(n+1)}={k+\ell+1\choose k+1}\zeta(k+\ell+1),
$$
where $k,\ell$ are nonnegative integrs with $k+\ell\geq 1$, and 
the modified Bell polynomials $P_m(x_1,x_2,\ldots,x_m)$ are 
defined by \cite{Chen,CC}
$\exp\left(\sum^\infty_{k=1}\frac{x_k}{k}z^k\right)
=\sum^\infty_{m=0}P_m(x_1,x_2,\ldots,x_m)z^m$.

Inspired by his work we extend his method to get 
a general explicitly expression of the form with Eq.\,(\ref{eq.11}).
The following is our main theorem.
\begin{mThm}
Let $M_{(\alpha_1,\ldots,\alpha_k)}=\zeta_n(m\alpha_1,\ldots,m\alpha_k;z)$
be the monomial basis for $\QSym$ with $M_0=1$.
For $u\in\QSym$ and $(s_1,\ldots,s_k)\in\mathbb N_0^k$ with $s_1+\cdots+s_k\geq 2$,
the $H$-function 
$$
\eta_{s_1,\ldots,s_k}(u)=\sum^\infty_{n=1}\frac{u}
{(n+z)^{s_1}(n+1+z)^{s_2}\cdots(n+k-1+z)^{s_k}}
$$
can be explicitly expressed as a linear combination of multiple Hurwitz zeta functions
and some speical values of $H_n^{(m)}(z)$.
\end{mThm}

For our convenience, we denote the $k$ repetitions of $a$ in the subscript of 
the $H$-function $\eta$ as $\eta_{\underbrace{a,\ldots,a}_k}=\eta_{a^k}$,
for example, $\eta_{1,1,0,2,2,2}=\eta_{1^2,0,2^3}$.

There are a lot of formulas related to harmonic numbers 
and reciprocal binomial coefficients of the form
$$
\sum^\infty_{n=1}\frac{F(H_n^{(m)},H_n^{(2m)},\ldots,H_n^{(\ell m)})}
{n^p{n+k\choose k}^q},
$$
derived in \cite{SXZ, Sof, Sof2, SS, XZZ}, for example \cite[Eq.\,(1.1)]{Sof2},
$$
\sum^\infty_{n=1}\frac{H_n^{(2)}}{{n+k\choose k}}
=\frac{k}{k-1}\left(\zeta(2)-H_{k-1}^{(2)}\right).
$$
A more complicated formula is given in \cite{SS}
$$
\sum^\infty_{n=1}\frac{H_n}{n^4{n+5\choose 5}}
=3\zeta(5)-\zeta(3)\zeta(2)-\frac{137}{48}\zeta(4)+\frac{12019}{1800}\zeta(3)
-\frac{874853}{216000}\zeta(2)+\frac{131891}{172800}.
$$
We give some more general explicit formulas to obtain all these above
identities as applications of our Main Theorem in the last section.

This paper is organized as follows.
In Section 2, we introduce some preliminaries about quasi-symmetric functions
and some basic lemmas. In Section 3, we first give the general formulas
for $\eta_{0^a,1,0^b,1}$ with nonnegative integers $a, b$. We give some
concret examples in the second part. In Section 4,
we give the general formulas for $\eta_{0^a,p}$ with nonnegative integers $a, p$
and $p\geq 2$. In the last part of this section we give examples using the formulas
of $\eta_{0^a,p}$. Finally, we conclude our main theorem in the end of Section 4.
In the last section we discuss some formulas related to harmonic numbers 
and reciprocal binomial coefficients.

\section{Quasi-Symmetric Functions}
Here we list some preliminaries about quasi-symmetric functions 
and symmetric functions \cite{Hof1, Mac, Sta}. 
Let $\mathbb Q[[x_1,x_2,\ldots]]$ be the set of formal power series of bounded degree. 
An element $u\in\mathbb Q[[x_1,x_2,\ldots]]$ such that the coefficient in $u$
of any monomial $x_{a_1}^{\alpha_1}\cdots x_{a_k}^{\alpha_k}$ with 
$a_1<a_2<\cdots<a_k$ is the same as that of $x_1^{\alpha_1}\cdots x_k^{\alpha_k}$
is called quasi-symmetric. Let \QSym be the subring of all the quasi-symmetric functions.
For any composition (ordered partition) $(\alpha_1,\ldots,\alpha_k)$ of $n$,
we denote $M_{(\alpha_1,\ldots,\alpha_k)}$ as the ``smallest'' quasi-symmetric function
containing $x_1^{\alpha_1}\cdots x_k^{\alpha_k}$.
It is convenient to denote $M_0=1$.

Let \Sym be the subring of $\mathbb Q[[x_1,x_2,\ldots]]$ 
that are invariant under permutations of the $x_i$.
We set $e_m$, $h_m$, and $p_m$ being the $m$-th elementary, complete homogeneous,
and power-sum symmetric polynomials, respectively.
They have associated generating functions
$$\begin{array}{lllllllll}
E(t)&:=&\displaystyle\sum^\infty_{j=0}e_jt^j 
     &=&\displaystyle\prod^\infty_{i=1}(1+tx_i),\\
H(t)&:=&\displaystyle\sum^\infty_{j=0}h_jt^j 
      &=&\displaystyle\prod^\infty_{i=1}\frac 1{1-tx_i} &=& E(-t)^{-1},\\
P(t)&:=&\displaystyle\sum^\infty_{j=1}p_jt^{j-1} 
      &=&\displaystyle\sum^\infty_{i=1}\frac{x_i}{1-tx_i} 
      &=&\displaystyle\frac{H'(t)}{H(t)}
      &=&\displaystyle\frac{E'(-t)}{E(-t)}.
\end{array}$$
Let $N_{n,k}$ be the sum of all the monomial symmetric functions corresponding to 
partitions of $n$ having length $k$. That is,
$$
N_{n,k}=\sum_{\alpha_1+\cdots+\alpha_k=n\atop \alpha_i\geq 1}M_{(\alpha_1,\ldots,\alpha_k)}.
$$
It is known that $p_m=N_{m,1}$ and $e_m=N_{m,m}$.
For our convenience, we let $N_{n,0}=0$,
for $n\geq 1$, and $N_{0,0}=1$. 

Let the modified Bell polynomials $P_m(x_1,x_2,\ldots,x_m)$ be 
defined by \cite{Chen,CC}
$$
\exp\left(\sum^\infty_{k=1}\frac{x_k}{k}z^k\right)
=\sum^\infty_{m=0}P_m(x_1,x_2,\ldots,x_m)z^m.
$$
The general explicit expression for $P_m$ is
$$
P_m(x_1,\ldots,x_m)=
\sum_{k_1+2k_2+\cdots+mk_m=m\atop k_i\geq 0}
\frac 1{k_1!\cdots k_m!}
\left(\frac{x_1}1\right)^{k_1}\cdots
\left(\frac{x_m}m\right)^{k_m}.
$$
Then we have \cite[Lemma 1]{CCE}
\begin{eqnarray*} 
e_k &=& P_k(p_1,-p_2,\ldots,(-1)^{k+1}p_k),\\  
h_k &=& P_k(p_1,p_2,\ldots,p_k).  
\end{eqnarray*}
And for $0\leq k\leq n$ \cite[Lemma 2]{Hof1},
\begin{equation}\label{eq.21}
e_kh_{n-k}=\sum^n_{j=k}{j\choose k}N_{n,j}.
\end{equation}

Let $x_i=1/(i+z)^{m}$, 
for $1\leq i\leq n$; $x_i=0$, for $i>n$, where $m$ is a positive integer. Then
$$
M_{(\alpha_1,\ldots,\alpha_k)}=\zeta_n(m\alpha_1,m\alpha_2,\ldots,m\alpha_k;z),\mbox{ and }
\zeta(m\alpha_1,m\alpha_2,\ldots,m\alpha_k;z)=\lim_{n\rightarrow\infty}M_{(\alpha_1,\ldots,\alpha_k)}.
$$
Thus the $k$-th power-sum, elementary, and complete homogeneous symmetric polynomials 
in \Sym are
\begin{eqnarray*}
p_k &=& H_n^{(km)}(z)\ =\ \sum^n_{\ell=1}\frac1{(\ell+z)^{km}},\\
e_k &=& P_k(H_n^{(m)}(z),-H_n^{(2m)}(z),\ldots, (-1)^{k+1}H_n^{(km)}(z)),\\
h_k &=& P_k(H_n^{(m)}(z),H_n^{(2m)}(z),\ldots,H_n^{(km)}(z)).\\
\end{eqnarray*}
For $u\in\QSym$ and $(s_1,\ldots,s_k)\in\mathbb N_0^k$ with $s_1+\cdots+s_k\geq 2$, we
define the $H$-function $\eta_{s_1,\ldots,s_k}:\QSym\rightarrow\mathbb R$ by
$$
\eta_{s_1,\ldots,s_k}(u)=\sum^\infty_{n=1}\frac{u}{(n+z)^{s_1}(n+1+z)^{s_2}\cdots(n+k-1+z)^{s_k}}.
$$ 
\begin{theorem}
$\eta_{s_1,\ldots,s_k}(u)$ converges for any $u\in\QSym$.
\end{theorem}
\begin{proof}
For any composition $I=(\alpha_1,\ldots,\alpha_k)$, we have
\begin{eqnarray*}
\lefteqn{\eta_{s_1,\ldots,s_k}(M_I)}\\
&=&\sum^\infty_{n=1}\frac{M_I}{(n+z)^{s_1}(n+1+z)^{s_2}\cdots(n+k-1+z)^{s_k}}\\
&=& \sum_{1\leq n_1<\cdots<n_k}\frac1{(n_1+z)^{m\alpha_1}\cdots(n_k+z)^{m\alpha_k}}
\sum_{n=n_k}^\infty\frac1{(n+z)^{s_1}\cdots(n+k-1+z)^{s_k}}.
\end{eqnarray*}
This is less than $\eta_2(M_I)$:
\begin{eqnarray}\label{eq.22}
\eta_2(M_I) &=&
\sum_{1\leq n_1<\cdots<n_k}\frac1{(n_1+z)^{m\alpha_1}\cdots(n_k+z)^{m\alpha_k}}
\sum_{n=n_k}^\infty\frac1{(n+z)^{2}}\\
&=&\zeta(m\alpha_1,\ldots,m\alpha_k+2;z)+\zeta(m\alpha_1,\ldots,m\alpha_k,2;z).\nonumber
\end{eqnarray}
Since the last terms are convergent, $\eta_{s_1,\ldots,s_k}(M_I)$ converges by the comparison theorem.
\QSym can be generated by $M_I$, therefore we complete the proof.
\end{proof}
The following is easily extended from \cite[Proposition 6]{Hof1}.
\begin{lemma}\label{lma.22}
Let $s_1,\ldots,s_k$ be a nonnegative integer sequence with $s_i,s_j\geq 1$ for $1\leq i<j\leq k$.
If $s_1+\cdots+s_k\geq 3$, then
$$
\eta_{s_1,\ldots,s_i,\ldots,s_j,\ldots,s_k}=\frac1{j-i}
\left(\eta_{s_1,\ldots,s_i,\ldots,s_j-1,\ldots,s_k}-\eta_{s_1,\ldots,s_i-1,\ldots,s_j,\ldots,s_k}\right).
$$
\end{lemma}
\begin{proof}
Using the following key equation 
$$
\frac1{(n+i-1+z)(n+j-1+z)}=\frac1{j-i}\left(\frac1{n+i-1+z}-\frac1{n+j-1+z}\right),
$$
we get the desired result. 
\end{proof}
Therefore $\eta_{s_1,\ldots,s_k}$ can be written as a linear combination of the following two 
types: $\eta_{0^a,p}$ and $\eta_{0^a,1,0^b,1}$, where $a$, $b$, $p$ are nonnegetive integers
and $p\geq 2$. For example,
$$
\eta_{2,3,2}=\eta_{0,1,1}-\eta_{1,1}+\frac14\eta_2+\eta_{0,3}-\frac14\eta_{0,0,2}.
$$
In order to evaluate these two types of the general formulas for 
$\eta_{0^a,p}$ and $\eta_{0^a,1,0^b,1}$, we need the following partial fraction decomposition:
\begin{lemma}\label{lma.23}
For positive integers $k, m$, and $a$, we have
\begin{equation}\label{eq.23}
\frac1{x^k(x+a)^m}=\sum^k_{\ell=1}{k-\ell+m-1\choose m-1}\frac{(-1)^{k-\ell}}{a^{m+k-\ell}x^\ell}
+\sum^m_{\ell=1}{k-\ell+m-1\choose k-1}\frac{(-1)^k}{a^{m+k-\ell}(x+a)^\ell}.
\end{equation}
\end{lemma}
\begin{proof}
We use the induction on $m+k$ to prove this lemma. For $m=k=1$, it is clearly that 
$$
\frac1{x(x+a)}=\frac1{ax}-\frac1{a(x+a)}.
$$
Assume that Eq.\,(\ref{eq.23}) is true for $2\leq m+k\leq n$. Thus 
\begin{eqnarray}\label{eq.24}
\lefteqn{\frac1{x^{k+1}(x+a)^m}=\frac1x\cdot\frac1{x^k(x+a)^m}} \\ \nonumber
&=&\sum^k_{\ell=1}{k-\ell+m-1\choose m-1}\frac{(-1)^{k-\ell}}{a^{m+k-\ell}x^{\ell+1}}
+\sum^m_{\ell=1}{k-\ell+m-1\choose k-1}\frac{(-1)^k}{a^{m+k-\ell}x(x+a)^\ell}.
\end{eqnarray}
The last equation is given by the inductive hypothese.
For $1\leq\ell\leq m$, we can apply the inductive hypothese again to get the following equation.
$$
\frac1{x(x+a)^\ell}=\frac1{a^\ell x}-\sum^\ell_{j=1}\frac1{a^{\ell+1-j}(x+a)^j}.
$$
Therefore
\begin{eqnarray*}
\lefteqn{\sum^m_{\ell=1}{k-\ell+m-1\choose k-1}\frac{(-1)^k}{a^{m+k-\ell}x(x+a)^\ell}}\\
&=&\frac{(-1)^k}{a^{m+k}x}\sum^m_{\ell=1}{k+m-\ell-1\choose k-1}
+(-1)^{k+1}\sum^m_{\ell=1}\sum^\ell_{j=1}\frac{{k+m-\ell-1\choose k-1}}{a^{m+k+1-j}(x+a)^j}\\
&=&\frac{(-1)^k}{a^{m+k}x}{k+m-1\choose m-1}
+(-1)^{k+1}\sum^m_{j=1}\frac{1}{a^{m+k+1-j}(x+a)^j}\sum^m_{\ell=j}{k+m-\ell-1\choose k-1}\\
&=& \frac{(-1)^k}{a^{m+k}x}{k+m-1\choose m-1}
+(-1)^{k+1}\sum^m_{j=1}\frac{{k+m-j\choose m-j}}{a^{m+k+1-j}(x+a)^j},
\end{eqnarray*}
where $\sum^m_{\ell=j}{k+m-\ell-1\choose k-1}={k+m-j\choose m-j}$ is obtained from
\cite[Eq.\,(5.9)]{GKP}.
Substituing this equation in Eq.\,(\ref{eq.24}) we have 
$$
\frac1{x^{k+1}(x+a)^m}
=\sum^{k+1}_{\ell=1}{k-\ell+m\choose m-1}\frac{(-1)^{k+1-\ell}}{a^{m+k+1-\ell}x^\ell}
+\sum^m_{\ell=1}{k-\ell+m\choose k}\frac{(-1)^{k+1}}{a^{m+k+1-\ell}(x+a)^\ell}.
$$
For the case $\frac1{x^k(x+a)^{m+1}}$, the proof is similarly. Therefore we omit them.
\end{proof}
\section{The general formula for $\eta_{0^a,1,0^b,1}$}
First we derive a general formula for $\eta_{0^a,1,0^b,1}$ with nonnegative integers $a$, $b$.
\begin{lemma}
For nonnegative integers $a$, $b$, we have 
\begin{equation}\label{eq.31}
\eta_{0^a,1,0^b,1} = \frac1{b+1}\sum^b_{r=0}\eta_{0^{r+a},1^2}.
\end{equation}
\end{lemma}
\begin{proof}
Fix the nonnegative integer $a$, the proof is used the mathematical induction on the number $b$.
We use Lemma \ref{lma.22} to the $H$-function 
$\eta_{0^a,1,0^b,1^2}$ with $i=a+1$, $j=a+b+3$, and we get the following equation:
$$
\eta_{0^a,1,0^b,1^2}=\frac1{b+2}\left(\eta_{0^a,1,0^b,1}-\eta_{0^{a+b+1},1^2}\right).
$$
On the other hand, we apply Lemma \ref{lma.22} again to the same $H$-function, but 
with $i=a+b+2$ and $j=a+b+3$:
$$
\eta_{0^a,1,0^b,1^2}=\eta_{0^a,1,0^b,1}-\eta_{0^a,1,0^{b+1},1}.
$$
Combining these two equations we have
$$
(b+2)\eta_{0^a,1,0^{b+1},1}=\eta_{0^{a+b+1},1^2}+(b+1)\eta_{0^a,1,0^b,1}.
$$
Now we apply the inductive hypothese and complete the proof.
\end{proof}

Therefore we need only to evaluate $\eta_{0^b,1^2}$.
\begin{theorem}\label{thm.32}
For any composition $I=(\alpha_1,\ldots,\alpha_k)$, we have
$$
\eta_{1,1}(M_I) = \zeta(m\alpha_1,\ldots,m\alpha_k+1;z).
$$
\end{theorem}
\begin{proof}
Since
\begin{equation}\label{eq.32}
\sum_{n=n_k}^\infty \frac1{(n+z)(n+1+z)}=\sum_{n=n_k}^\infty
\left(\frac1{n+z}-\frac1{n+1+z}\right)=\frac1{n_k+z},
\end{equation}
we have
\begin{eqnarray*}
\eta_{1,1}(M_I) &=& \sum_{1\leq n_1<\cdots<n_k}\frac1{(n_1+z)^{m\alpha_1}\cdots(n_k+z)^{m\alpha_k}}
\sum^\infty_{n=n_k}\frac1{(n+z)(n+1+z)} \\
&=&\sum_{1\leq n_1<\cdots<n_k}\frac1{(n_1+z)^{m\alpha_1}\cdots(n_k+z)^{m\alpha_k+1}}\\
&=& \zeta(m\alpha_1,\ldots,m\alpha_{k-1},m\alpha_k+1;z).
\end{eqnarray*}
\end{proof}
\begin{theorem}\label{thm.33}
Given a positive integer $b$ and a composition $I=(\alpha_1,\ldots,\alpha_k)$, we have
\begin{enumerate}
\item if $I=(1)$, then $\eta_{0^b,1^2}(M_I)=\displaystyle\frac{(-1)^{m-1}}{b^m}H_b(z)
+\sum^{m-2}_{r=0}\frac{(-1)^r}{b^{r+1}}\zeta(m-r;z)$;
\item if $I\neq(1)$ and $\alpha_k=1$, then
$\eta_{0^b,1^2}(M_I)=\displaystyle\frac{(-1)^{m-1}}{b^m}\sum^b_{r=1}\eta_{0^r,1^2}(M_{(\alpha_1,\ldots,\alpha_{k-1})})
+\displaystyle\sum^{m-2}_{r=0}\frac{(-1)^r}{b^{r+1}}\zeta(m\alpha_1,\ldots,m\alpha_{k-1},m-r;z)$,
\item otherwise (that is, $\alpha_k\geq 2$), then \\
$\eta_{0^b,1^2}(M_I)=
\displaystyle\frac{(-1)^{m}}{b^m}\eta_{0^b,1^2}(M_{(\alpha_1,\ldots,\alpha_k-1)})
+\displaystyle\sum^{m-1}_{r=0}\frac{(-1)^r}{b^{r+1}}\zeta(m\alpha_1,\ldots,m\alpha_k-r;z)$.
\end{enumerate}
\end{theorem}
\begin{proof}
If $I=(1)$, then by Eq.\,(\ref{eq.32})
$$
\eta_{0^b,1^2}(M_{(1)})=\sum_{n=1}^\infty\frac1{(n+z)^m(n+b+z)}.
$$
For a pair of positive integers $j,k$, using Lemma \ref{lma.23} we have the partial fraction decomposition
\begin{equation}\label{eq.33}
\frac1{(n+z)^{j}(n+k+z)}=\frac{(-1)^j}{k^{j}(n+k+z)}+\sum^{j-1}_{r=0}\frac{(-1)^r}{k^{r+1}(n+z)^{j-r}}.
\end{equation}
Thus 
\begin{eqnarray*}
\eta_{0^b,1^2}(M_{(1)}) &=& 
\frac{(-1)^{m-1}}{b^m}\sum^\infty_{n=1}\left(\frac1{n+z}-\frac1{n+b+z}\right)
+\sum^{m-2}_{r=0}\frac{(-1)^r}{b^{r+1}}\sum^\infty_{n=1}\frac1{(n+z)^{m-r}}\\
&=&\frac{(-1)^{m-1}}{b^m}H_b(z)
+\sum^{m-2}_{r=0}\frac{(-1)^r}{b^{r+1}}\zeta(m-r;z).
\end{eqnarray*}
If $I\neq(1)$ and $\alpha_k=1$, then 
\begin{eqnarray*}
\eta_{0^b,1^2}(M_I) &=& \sum_{1\leq n_1<\cdots<n_k}\frac1{(n_1+z)^{m\alpha_1}\cdots
(n_{k-1}+z)^{m\alpha_{k-1}}(n_k+z)^m(n_k+b+z)} \\
&=&\sum_{1\leq n_1<\cdots<n_{k-1}}\frac1{(n_1+z)^{m\alpha-1}\cdots(n_{k-1}+z)^{m\alpha_{k-1}}}\\
&&\qquad\qquad\qquad\times
\sum^\infty_{n_k=n_{k-1}+1}\frac1{(n_k+z)^m(n_k+b+z)}.
\end{eqnarray*}
We apply Eq.\,(\ref{eq.33}) to the last term of the above equation, then
\begin{eqnarray*}
\eta_{0^b,1^2}(M_I) &=& 
\frac{(-1)^{m-1}}{b^m}\sum^b_{r=1}\sum_{1\leq n_1<\cdots<n_{k-1}}
\frac1{(n_1+z)^{m\alpha_1}\cdots(n_{k-1}+z)^{m\alpha_{k-1}}(n_{k-1}+r+z)}\\
&&+\sum^{m-2}_{r=0}\frac{(-1)^r}{b^{r+1}}
\sum_{1\leq n_1<\cdots<n_k}\frac1{(n_1+z)^{m\alpha_1}\cdots(n_{k-1}+z)^{m\alpha_{k-1}}(n_k+z)^{m-r}}\\
&=&\frac{(-1)^{m-1}}{b^m}\sum^b_{r=1}\eta_{0^r,1^2}(M_{(\alpha_1,\ldots,\alpha_{k-1})})
+\sum^{m-2}_{r=0}\frac{(-1)^r}{b^{r+1}}\zeta(m\alpha_1,\ldots,m\alpha_{k-1},m-r;z).
\end{eqnarray*}
The final case is $\alpha_k\geq 2$. We have
$$
\eta_{0^b,1^2}(M_I)=
\sum_{1\leq n_1<\cdots<n_k}\frac1{(n_1+z)^{m\alpha_1}\cdots(n_k+z)^{m(\alpha_k-1)}
(n_k+z)^m(n_k+b+z)}.
$$
We substitute Eq.\,(\ref{eq.33}) into the above partial fraction, we have
\begin{eqnarray*}
\eta_{0^b,1^2}(M_I) &=&
\frac{(-1)^m}{b^m}\sum_{1\leq n_1<\cdots<n_k}\frac1{(n_1+z)^{m\alpha_1}\cdots(n_k+z)^{m(\alpha_k-1)}
(n_k+b+z)}\\
&&+\sum^{m-1}_{r=0}\frac{(-1)^r}{b^{r+1}}\sum_{1\leq n_1<\cdots<n_k}
\frac1{(n_1+z)^{m\alpha_1}\cdots(n_k+z)^{m\alpha_k-r}}\\
&=&\frac{(-1)^m}{b^m}\eta_{0^b,1^2}(M_{(\alpha_1,\ldots,\alpha_k-1)})
+\sum^{m-1}_{r=0}\frac{(-1)^r}{b^{r+1}}\zeta(m\alpha_1,\ldots,m\alpha_k-r;z).
\end{eqnarray*}
We thus complete the proof.
\end{proof}
Combing Theorem \ref{thm.32} and Theorem \ref{thm.33} 
we know that $\eta_{0^a,1,0^b,1}(M_I)$ can be written as 
a linear combination of multiple Hurwitz zeta functions
and some speical values of $H_n^{(m)}(z)$.
\begin{theorem}\label{thm.34}
For nonnegative integers $a,b$ and any composition $I=(\alpha_1,\ldots,\alpha_k)$,
$$
\sum_{n=1}^\infty\frac{M_I}{(n+a+z)(n+a+b+1+z)}=\eta_{0^a,1,0^b,1}(M_I)
=\frac1{b+1}\sum^b_{r=0}\eta_{0^{r+a},1^2}(M_I)
$$
can be expressed as a linear combination of multiple Hurwitz zeta functions 
and some speical values of $H_n^{(m)}(z)$.
\end{theorem}
It is worth to note that Theorem \ref{thm.33} (1) is the following identitiy:
\begin{equation}\label{eq.34}
\sum^\infty_{n=1}\frac{H_n^{(m)}(z)}{(n+b+z)(n+b+1+z)}
=\frac{(-1)^{m-1}}{b^m}H_b(z)
+\sum^{m-2}_{r=0}\frac{(-1)^r}{b^{r+1}}\zeta(m-r;z).
\end{equation}
When $b=1$ and $z=0$ this identity is appeared in \cite[Corollary 1]{Hof1}.
Here we list the identity with $z=-1/2$:
\begin{equation}\label{eq.35}
\sum^\infty_{n=1}\frac{O_n^{(m)}}{(2n+2b-1)(2n+2b+1)}
=\frac{(-1)^{m-1}O_b}{2^{m+1}b^m}+\sum^{m-2}_{\ell=0}
\frac{(-1)^\ell t(m-\ell)}{2^{\ell+2}b^{\ell+1}}.
\end{equation}
The following are the identities with $b=1$ and $m=1,2,3,4$:
\begin{eqnarray*}
\sum^\infty_{n=1}\frac{O_n}{(2n+1)(2n+3)} &=& \frac14,\\
\sum^\infty_{n=1}\frac{O_n^{(2)}}{(2n+1)(2n+3)} &=& \frac{t(2)}4-\frac18,\\
\sum^\infty_{n=1}\frac{O_n^{(3)}}{(2n+1)(2n+3)} &=& \frac{t(3)}4-\frac{t(2)}4+\frac1{16},\\
\sum^\infty_{n=1}\frac{O_n^{(4)}}{(2n+1)(2n+3)} &=& \frac{t(4)}4-\frac{t(3)}8
+\frac{t(2)}{16}-\frac1{32}.
\end{eqnarray*}
\begin{proposition}\label{pro.35}
For positive integers $k$, $m$, $b$, and a real number $z\in(-1,0]$, we have
\begin{eqnarray*}
\lefteqn{\sum^\infty_{n=1}\frac{P_k(H_n^{(m)}(z),-H_n^{(2m)}(z),
\ldots,(-1)^{k+1}H_n^{(km)}(z))}{(n+b+z)(n+b+1+z)}}\\
&=&\sum_{1\leq r_{k-1}\leq r_{k-2}\leq \cdots\leq r_0=b}
\frac{(-1)^{k(m-1)}}{\prod^{k-1}_{p=0}r_p^m}H_{r_{k-1}}(z)\\
&&+\sum^{k-1}_{j=0}\sum_{1\leq r_j\leq\cdots\leq r_0=b}\frac{(-1)^{j(m-1)}}{\prod^{j-1}_{p=0}r_p^m}
\sum^{m-2}_{\ell=0}\frac{(-1)^{\ell}}{r_j^{\ell+1}}\zeta(\{m\}^{k-j-1},m-\ell;z).
\end{eqnarray*}
\end{proposition}
\begin{proof}
The left-hand side of the equation is $\eta_{0^b,1,1}(e_k)$.
If $k\geq 2$, then we use Theorem \ref{thm.33} (2) and we have the following recursive relation
$$
\eta_{0^b,1,1}(e_k)=\sum^{m-2}_{r=0}\frac{(-1)^r}{b^{r+1}}\zeta(\{m\}^{k-1},m-r;z)
+\frac{(-1)^{m-1}}{b^m}\sum^b_{r=1}\eta_{0^r,1,1}(e_{k-1}).
$$
Solve this recursive relation we will reach the following equation
\begin{eqnarray*}
\lefteqn{\eta_{0^b,1,1}(e_k)=\sum_{1\leq r_{k-1}\leq r_{k-2}\leq \cdots\leq r_0=b}
\frac{(-1)^{(k-1)(m-1)}}{\prod^{k-2}_{p=0}r_p^m}\eta_{0^{r_{k-1}},1,1}(e_1)}\\
&&+\sum^{k-2}_{j=0}\sum_{1\leq r_j\leq \cdots\leq r_0=b}
\frac{(-1)^{j(m-1)}}{\prod^{j-1}_{p=0}r_p^m}\sum^{m-2}_{\ell=0}\frac{(-1)^\ell}{r_j^{\ell+1}}
\zeta(\{m\}^{k-j-1},m-\ell;z).
\end{eqnarray*}
We apply Theorem \ref{thm.33} (1), then the desired formula will be obtained.
Note that this formula is also true for $k=1$.
\end{proof}
We set $z=0$ and $m=1$ in the above proposition, the following identity is obtained.
\begin{example}
For a pair of positive integers $b$ and $k$, we have
\begin{equation}\label{eq.36}
\sum^\infty_{n=1}\frac{P_k(H_n,-H_n^{(2)},\ldots,(-1)^{k+1}H_n^{(k)})}{(n+b)(n+b+1)}
=\frac{\zeta^\star_b(\{1\}^k)}b=\frac1b\sum^b_{\ell=1}{b\choose\ell}\frac{(-1)^{\ell+1}}{\ell^k},
\end{equation}
where the last right-hand side identity can be found in \cite{Bat}.
\end{example}
We list the identities for $b=1,2,3$:
\begin{eqnarray*}
\sum^\infty_{n=1}\frac{P_k(H_n,-H_n^{(2)},\ldots,(-1)^{k+1}H_n^{(k)})}{(n+1)(n+2)} &=& 1,
\quad\mbox{\cite[Corollary 3]{Cho}, \cite[Eq.(8)]{Hof1}}\\
\sum^\infty_{n=1}\frac{P_k(H_n,-H_n^{(2)},\ldots,(-1)^{k+1}H_n^{(k)})}{(n+2)(n+3)} &=& 1-\frac1{2^{k+1}},\\
\sum^\infty_{n=1}\frac{P_k(H_n,-H_n^{(2)},\ldots,(-1)^{k+1}H_n^{(k)})}{(n+3)(n+4)} &=&
1-\frac1{2^k}+\frac1{3^{k+1}}.
\end{eqnarray*}
We set $z=0$, $m=2$, and $b=1$ in Proposition \ref{pro.35}, the following identity is obtained.
\begin{example}
For any positive integer $k$, we have
\begin{equation}\label{eq.37}
\sum^\infty_{n=1}\frac{P_k(H_n^{(2)},-H_n^{(4)},\ldots,(-1)^{k+1}H_n^{(2k)})}{(n+1)(n+2)} =
(-1)^k+\sum^k_{j=1}(-1)^{k-j}\zeta(\{2\}^j).
\end{equation}
\end{example}
We list the identities for k=1,2,3:
\begin{eqnarray*}
\sum^\infty_{n=1}\frac{H_n^{(2)}}{(n+1)(n+2)} &=& -1+\zeta(2),\\
\sum^\infty_{n=1}\frac{(H_n^{(2)})^2-H_n^{(4)}}{2(n+1)(n+2)}&=&1-\zeta(2)+\zeta(2,2),\\
\sum^\infty_{n=1}\frac{(H_n^{(2)})^3-3H_n^{(2)}H_n^{(4)}+2H_n^{(6)}}{6(n+1)(n+2)} &=&
-1+\zeta(2)-\zeta(2,2)+\zeta(2,2,2).
\end{eqnarray*}
We set $z=-1/2$, $m=1$, and $b=1$ in Proposition \ref{pro.35}, we have
\begin{example}
For any positive integer $k$, the following identity holds:
\begin{equation}\label{eq.38}
\sum^\infty_{n=1}\frac{P_k(O_n,-O_n^{(2)},\ldots,(-1)^{k+1}O_n^{(k)})}{(2n+1)(2n+3)}
=\frac1{2^{k+1}}.
\end{equation}
\end{example}
The following we list the identities which we set $z=-1/2$, $m=1$ in Proposition \ref{pro.35},
for $k=1,2$.
\begin{eqnarray}\label{eq.39}
\sum^\infty_{n=1}\frac{O_n}{(2n+2b-1)(2n+2b+1)} &=& \frac{O_b}{4b},\\
\sum^\infty_{n=1}\frac{O_n^2-O_n^{(2)}}{(2n+2b-1)(2n+2b+1)}&=&\frac1{4b}\sum^b_{r=1}\frac{O_r}r. \label{eq.310}
\end{eqnarray}
We set $z=-1/2$, $m=2$, and $b=1$ in Proposition \ref{pro.35}, the following identity is obtained.
\begin{example}
For any positive integer $k$, we have
$$
2^{2k+1}\sum^\infty_{n=1}\frac{P_k(O_n^{(2)},-O_n^{(4)},\ldots,(-1)^{k+1}O_n^{(2k)})}{(2n+1)(2n+3)}
=(-1)^k+\sum^k_{j=1}(-1)^{k-j}2^{2j-1}t(\{2\}^j).
$$
\end{example}
We list the identities for $k=1,2,3$:
\begin{eqnarray*}
8\sum^\infty_{n=1}\frac{O_n^{(2)}}{(2n+1)(2n+3)} &=&-1+2t(2),\\
32\sum^\infty_{n=1}\frac{(O_n^{(2)})^2-O_n^{(4)}}{2(2n+1)(2n+3)} &=& 
1-2t(2)+8t(2,2),\\
128\sum^\infty_{n=1}\frac{(O_n^{(2)})^3-3O_n^{(2)}O_n^{(4)}+2O_n^{(6)}}{6(2n+1)(2n+3)} &=&
-1+2t(2)-8t(2,2)+32t(2,2,2).
\end{eqnarray*}
In the end of this section we give another example: $\eta_{0,1,1}(e_1h_1)$ with $m=1$.
\begin{proposition}
For any positive integer $b$ and a real number $z\in(-1,0]$, we have
$$
\sum^\infty_{n=1}\frac{H_n(z)^2}{(n+b+z)(n+b+1+z)}
=\frac{\zeta(2;z)}b-\frac{H_b(z)}{b^2}+\frac2b\sum^b_{r=1}\frac{H_r(z)}r.
$$
\end{proposition}
\begin{proof}
Since $h_1e_1=p_2+2e_2$ (ref. Eq.\,(\ref{eq.21})), we apply Theorem 3 with $m=1$ to get
\begin{eqnarray*}
\eta_{0^b,1,1}(e_1h_1) &=& \eta_{0^b,1,1}(p_2)+2\eta_{0^b,1,1}(e_2) \\
&=& \frac{\zeta(2;z)}b-\frac{H_b(z)}{b^2}+\frac2b\sum^b_{r=1}\frac{H_r(z)}r.
\end{eqnarray*}
It is known that 
$$
\eta_{0^b,1,1}(e_1h_1)=\sum^\infty_{n=1}\frac{H_n(z)^2}{(n+b+z)(n+b+1+z)},
$$
thus we complete the proof.
\end{proof}
If we set $z=0$ and $z=-1/2$ in the above proposition, then we have
\begin{eqnarray}\label{eq.311}
\sum^\infty_{n=1}\frac{H_n^2}{(n+b)(n+b+1)}
&=&\frac{\zeta(2)}{b}+\frac{bH_b^{(2)}+bH_b^2-H_b}{b^2},\\
\sum^\infty_{n=1}\frac{O_n^2}{(2n+2b-1)(2n+2b+1)}
&=&\frac{t(2)}{4b}-\frac{O_b}{8b^2}+\frac1{4b}\sum^b_{r=1}\frac{O_r}r.\label{eq.312}
\end{eqnarray}
The first identity with $b=1$ appears in 
\cite[Corollary 5]{Cho}, \cite[Theorem 1]{Hof1}.

\section{The general formula for $\eta_{0^a,p}$}
In this section, we derive the general formula for $\eta_{0^a,p}$ for nonnegative integer $a$.
First we have the following trivial results.
\begin{theorem}\label{thm.41}
For any compositiion $I=(\alpha_1,\ldots,\alpha_k)$, we have
\begin{eqnarray*}
\eta_p(M_I) &=& \zeta(m\alpha_1,\ldots,m\alpha_k+p;z)+\zeta(m\alpha_1,\ldots,m\alpha_k,p;z),\\
\eta_{0,p}(M_I) &=& \zeta(m\alpha_1,\ldots,m\alpha_k,p;z),
\end{eqnarray*}
where $p\geq 2$ is a positive integer.
\end{theorem}
Let $a\geq 0$ and $p\geq 1$ be two integers, we define a function:
\begin{equation}\label{eq.41}
T_{0^ap}=\left\{\begin{array}{ll}
\eta_{0^a,1^2}, &\mbox{ if }p=1;\\
\eta_{0^a,p}-\eta_{0^{a+1},p},&\mbox{ if }p\geq 2.
\end{array}\right.
\end{equation}
\begin{theorem}\label{thm.42}
Let $a\geq 1$, $p\geq 2$ be integers, and $I=(\alpha_1,\ldots,\alpha_k)$ be any composition 
with $k\geq 2$. Then 
\begin{eqnarray}\label{eq.42}
T_{0^ap}(M_I) &=&\sum^{m\alpha_k}_{\ell=2}{m\alpha_k+p-\ell-1\choose p-1}
\frac{(-1)^{m\alpha_k-\ell}}{a^{p+m\alpha_k-\ell}}\zeta(m\alpha_1,\ldots,m\alpha_{k-1},\ell;z) \\
&&+\sum^p_{\ell=2}{m\alpha_k+p-\ell-1\choose m\alpha_k-1} \nonumber
\frac{(-1)^{m\alpha_k}}{a^{p+m\alpha_k-\ell}}\zeta(m\alpha_1,\ldots,m\alpha_{k-1},\ell;z) \\
&&+\sum^p_{\ell=1}{m\alpha_k+p-\ell-1\choose m\alpha_k-1}
\frac{(-1)^{m\alpha_k-1}}{a^{p+m\alpha_k-\ell}} \nonumber
\sum^a_{j=1}T_{0^j\ell}(M_{(\alpha_1,\ldots,\alpha_{k-1})}).
\end{eqnarray}
\end{theorem}
\begin{proof}
Since 
$$
T_{0^ap}(M_I)=\sum_{1\leq n_1<\cdots<n_k}
\frac1{(n_1+z)^{m\alpha_1}\cdots(n_k+z)^{m\alpha_k}(n_k+a+z)^p},
$$
we apply Lemma \ref{lma.23} to the last two factors $1/((n_k+z)^{m\alpha_k}(n_k+a+z)^p)$. Then
\begin{eqnarray*}
T_{0^ap}(M_I) &=& \sum_{1\leq n_1<\cdots<n_k}
\frac1{(n_1+z)^{m\alpha_1}\cdots(n_{k-1}+z)^{m\alpha_{k-1}}}\\
&&\quad\times\left[
{m\alpha_k+p-2\choose p-1}\frac{(-1)^{m\alpha_k-1}}{a^{p+m\alpha_k-1}}\left(
\frac1{n_k+z}-\frac1{n_k+z+a}\right) \right.\\
&&\qquad\qquad+\sum^{m\alpha_k}_{\ell=2}{m\alpha_k-\ell+p-1\choose p-1}
\frac{(-1)^{m\alpha_k-\ell}}{a^{p+m\alpha_k-\ell}(n_k+z)^\ell}\\
&&\qquad\qquad\qquad\left.+\sum^p_{\ell=2}{m\alpha_k-\ell+p-1\choose m\alpha_k-1}
\frac{(-1)^{m\alpha_k}}{a^{p+m\alpha_k-\ell}(n_k+z+a)^\ell}\right]. 
\end{eqnarray*}
There are three summands. The first summand is 
$$
{m\alpha_k+p-2\choose p-1}\frac{(-1)^{m\alpha_k-1}}{a^{p+m\alpha_k-1}}
\sum_{1\leq n_1<\cdots<n_{k-1}}
\frac1{(n_1+z)^{m\alpha_1}\cdots(n_{k-1}+z)^{m\alpha_{k-1}}}
\sum^a_{\ell=1}\frac1{n_{k-1}+\ell+z}.
$$
This can be written as the following form
$$
{m\alpha_k+p-2\choose p-1}\frac{(-1)^{m\alpha_k-1}}{a^{p+m\alpha_k-1}}
\left(\eta_{0,1^2}+\eta_{0^2,1^2}+\cdots+\eta_{0^a,1^2}\right)
\left(M_{(\alpha_1,\ldots,\alpha_{k-1})}\right).
$$
It can be easily to see that the second summand is
$$
\sum^{m\alpha_k}_{\ell=2}{m\alpha_k-\ell+p-1\choose p-1}
\frac{(-1)^{m\alpha_k-\ell}}{a^{p+m\alpha_k-\ell}}\,
\zeta(m\alpha_1,\ldots,m\alpha_{k-1},\ell;z).
$$
Since 
$$
\sum_{n_k=n_{k-1}+1}^\infty\frac1{(n_k+z+a)^\ell}
=\sum_{n_k=n_{k-1}+1}^\infty\frac1{(n_k+z)^\ell}-\sum^a_{j=1}\frac1{(n_{k-1}+j+z)^\ell},
$$
We can write the third summand as 
\begin{eqnarray*}
\lefteqn{\sum^p_{\ell=2}{m\alpha_k-\ell+p-1\choose m\alpha_k-1}
\frac{(-1)^{m\alpha_k}}{a^{p+m\alpha_k-\ell}}\zeta(m\alpha_1,\ldots,m\alpha_{k-1},\ell;z)} \\
&-&\sum^p_{\ell=2}{m\alpha_k-\ell+p-1\choose m\alpha_k-1}
\frac{(-1)^{m\alpha_k}}{a^{p+m\alpha_k-\ell}}
\left(T_{0\ell}+T_{0^2\ell}+\cdots+T_{0^a\ell}\right)
\left(M_{(\alpha_1,\ldots,\alpha_{k-1})}\right).
\end{eqnarray*}
Combining these three results together we complete the proof.
\end{proof}

\begin{theorem}\label{thm.43}
For a pair of positive integers $a,p$, we have
\begin{eqnarray}\label{eq.43}
\lefteqn{T_{0^ap}(p_k)}\\ \nonumber
&=& \sum^{mk}_{\ell=2}{mk+p-\ell-1\choose p-1}\frac{(-1)^{mk-\ell}}{a^{p+mk-\ell}}\zeta(\ell;z) 
+\sum^p_{\ell=2}{mk+p-\ell-1\choose mk-1}\frac{(-1)^{mk}}{a^{p+mk-\ell}}\zeta(\ell;z)\\
&&+\sum^p_{\ell=1}{mk+p-\ell-1\choose mk-1}\frac{(-1)^{mk-1}}{a^{p+mk-\ell}}H_a^{(\ell)}(z).\nonumber
\end{eqnarray}
\end{theorem}
\begin{proof}
The proof is similar as the above theorem, but only to note that 
$\sum^a_{j=1}T_{0^j\ell}(M_0)=H_a^{(\ell)}(z)$.
Then we get the result.
\end{proof}
It is clearly that $\eta_{0^a,p}$ can be solved by using Theorem \ref{thm.41}, 
Theorem \ref{thm.42}, and Theorem \ref{thm.43}.
Thus we conclude that the following theorem.
\begin{theorem}\label{thm.44}
For a pair of integers $a\geq 0$, $p\geq 2$, and any composition $I=(\alpha_1,\ldots,\alpha_k)$,
$$
\eta_{0^a,p}(M_I)=\sum^\infty_{n=1}\frac{M_I}{(n+a+z)^p}
$$
can be expressed as a linear combination of multiple Hurwitz zeta functions 
and some speical values of $H_n^{(m)}(z)$.
\end{theorem}

For $p\geq 2$, by Theorem \ref{thm.41} we have $\eta_{0,p}(p_k)=\zeta(mk,p;z)$. Therefore we have 
a formula to evaluate $\eta_{0,0,p}(p_k)$ by using Eq.\,(\ref{eq.43}).
\begin{corollary}
For a positive integer $p\geq 2$, we have
\begin{eqnarray*}
\lefteqn{\eta_{0,0,p}(p_k)}\\ \nonumber
&=& \zeta(mk,p;z)-\sum^{mk}_{\ell=2}{mk+p-\ell-1\choose p-1}(-1)^{mk-\ell}\zeta(\ell;z) \\
&&-\sum^p_{\ell=2}{mk+p-\ell-1\choose mk-1}(-1)^{mk}\zeta(\ell;z)
+\sum^p_{\ell=1}{mk+p-\ell-1\choose mk-1}\frac{(-1)^{mk}}{(1+z)^\ell}.\nonumber
\end{eqnarray*}
\end{corollary}
It is known that the function $\eta_{0,0,p}(p_k)$ 
is $\sum^\infty_{n=1}\frac{H_n^{(mk)}(z)}{(n+2+z)^p}$.
Therefore this corollary gives the following identity.
\begin{eqnarray*}
\lefteqn{\sum^\infty_{n=1}\frac{H_n^{(k)}(z)}{(n+2+z)^p}}\\
&=& \zeta(k,p;z)-\sum^{k}_{\ell=2}{k+p-\ell-1\choose p-1}(-1)^{k-\ell}\zeta(\ell;z) \\
&&-\sum^p_{\ell=2}{k+p-\ell-1\choose k-1}(-1)^{k}\zeta(\ell;z)
+\sum^p_{\ell=1}{k+p-\ell-1\choose k-1}\frac{(-1)^{k}}{(1+z)^\ell}.
\end{eqnarray*}
We list the formula with $z=0$ in the following. For $k\geq 1$ and $p\geq 2$, we have
\begin{eqnarray}\label{eq.44}
\sum^\infty_{n=1}\frac{H_n^{(k)}}{(n+2)^p}
&=& \zeta(k,p)-\sum^{k}_{\ell=2}{k+p-\ell-1\choose p-1}(-1)^{k-\ell}\zeta(\ell) \\
&&-\sum^p_{\ell=2}{k+p-\ell-1\choose k-1}(-1)^{k}\zeta(\ell)
+(-1)^k{k+p-1\choose p-1}.\nonumber
\end{eqnarray}
The last summation is used \cite[Eq.\,(5.9)]{GKP} 
$\sum^p_{\ell=1}{k+p-\ell-1\choose k-1}={k+p-1\choose p-1}$ to simplify.

The following we give an evaluation of $\eta_{0,0,2}(e_k)$ using our theorems.
Let $a=1$ and $p=2$ in Theorem \ref{thm.42}, we obtain
\begin{eqnarray*}
T_{02}(e_k)&=&
\sum^m_{\ell=2}(-1)^{m-\ell}(m+1-\ell)\zeta(\{m\}^{k-1},\ell;z)
+(-1)^m\zeta(\{m\}^{k-1},2;z) \\
&&\qquad+(-1)^{m-1}mT_{01}(e_{k-1})+(-1)^{m-1}T_{02}(e_{k-1}).
\end{eqnarray*}
The term $T_{01}(e_{k-1})$ is $\eta_{0,1,1}(e_{k-1})$. 
Applying Proposition \ref{pro.35} with $b=1$ we have
$$
\eta_{0,1,1}(e_k)=
\frac{(-1)^{k(m-1)}}{1+z}+\sum^{k-1}_{j=0}(-1)^{j(m-1)}\sum^m_{\ell=2}
(-1)^{m-\ell}\zeta(\{m\}^{k-j-1},\ell;z).
$$
Substitue it into the above identity and solve the recursive relation to the final step:
\begin{eqnarray*}
\lefteqn{T_{02}(e_k)=(-1)^{(k-1)(m-1)}T_{02}(e_1)}\\
&+& \sum^{k-1}_{j=1}(-1)^{(j-1)(m-1)}\left[
\sum^m_{\ell=2}(-1)^{m-\ell}(m+1-\ell)\zeta(\{m\}^{k-j},\ell;z)+(-1)^m\zeta(\{m\}^{k-j},2;z) \right.\\
&+&\left.\frac{(-1)^{(m-1)(k-j+1)}m}{1+z}
+m\cdot\sum^{k-j-1}_{r=0}(-1)^{(r+1)(m-1)}\sum^m_{\ell=2}(-1)^{m-\ell}
\zeta(\{m\}^{k-r-j-1},\ell;z)\right].
\end{eqnarray*}
Because $e_1=p_1$, we apply Theorem \ref{thm.43} with $a=1$, $p=2$, and $k=1$:
$$
T_{02}(e_1)=\sum^m_{\ell=2}(-1)^{m-\ell}(m+1-\ell)\zeta(\ell;z)+
(-1)^m\zeta(2;z)+\frac{(-1)^{m-1}m}{1+z}+\frac{(-1)^{m-1}}{(1+z)^2}.
$$
Substitue it into the last formula of $T_{02}(e_k)$:
\begin{eqnarray*}
\lefteqn{T_{02}(e_k)=\frac{(-1)^{k(m-1)}}{(1+z)^2}}\\
&+& \sum^{k-1}_{j=1}(-1)^{(j-1)(m-1)}\left[
\sum^m_{\ell=2}(-1)^{m-\ell}(m+1-\ell)\zeta(\{m\}^{k-j},\ell;z)+(-1)^m\zeta(\{m\}^{k-j},2;z) \right.\\
&+&\left.\frac{(-1)^{(m-1)(k-j+1)}m}{1+z}\right]
+m\cdot\sum^{k-2}_{j=0}\sum^{j}_{r=0}(-1)^{(k-r-1)(m-1)}\sum^m_{\ell=2}(-1)^{m-\ell}
\zeta(\{m\}^{r},\ell;z).
\end{eqnarray*}
If we recombine the last summation, then we have a more simple form
$$
m\cdot\sum^{k-2}_{j=0}(k-1-j)(-1)^{(k-1-j)(m-1)}\sum^m_{\ell=2}(-1)^{m-\ell}
\zeta(\{m\}^j,\ell;z).
$$
Therefore we get the final form:
\begin{eqnarray*}
\lefteqn{T_{02}(e_k)=\frac{(-1)^{k(m-1)}}{(1+z)^2}}\\
&+& \sum^{k-1}_{j=1}(-1)^{(j-1)(m-1)}\left[
\sum^m_{\ell=2}(-1)^{m-\ell}(m+1-\ell)\zeta(\{m\}^{k-j},\ell;z)+(-1)^m\zeta(\{m\}^{k-j},2;z) \right.\\
&+&\left.\frac{(-1)^{(m-1)(k-j+1)}m}{1+z}\right]
+m\cdot\sum^{k-2}_{j=0}(k-1-j)(-1)^{(k-1-j)(m-1)}\sum^m_{\ell=2}(-1)^{m-\ell}
\zeta(\{m\}^j,\ell;z).
\end{eqnarray*}
Now $T_{02}(e_k)=\eta_{0,2}(e_k)-\eta_{0,0,2}(e_k)$ 
and from Theorem \ref{thm.41} $\eta_{0,2}(e_k)=\zeta(\{m\}^{k},2;z)$,
we have 
\begin{eqnarray}\label{eq.45}
\lefteqn{\sum^\infty_{n=1}\frac{P_k(H_n^{(m)}(z),-H_n^{(2m)}(z),\ldots,(-1)^{k+1}H_n^{(km)}(z))}{(n+2+z)^2}
=\zeta(\{m\}^k,2;z)-\frac{(-1)^{k(m-1)}}{(1+z)^2}}\\
&-& \sum^{k-1}_{j=1}(-1)^{(j-1)(m-1)}\left[  \nonumber
\sum^m_{\ell=2}(-1)^{m-\ell}(m+1-\ell)\zeta(\{m\}^{k-j},\ell;z)+(-1)^m\zeta(\{m\}^{k-j},2;z) \right.\\
&+&\left.\frac{(-1)^{(m-1)(k-j+1)}m}{1+z}\right]
-m\cdot\sum^{k-2}_{j=0}(k-1-j)(-1)^{(k-1-j)(m-1)}\sum^m_{\ell=2}(-1)^{m-\ell}
\zeta(\{m\}^j,\ell;z). \nonumber
\end{eqnarray}
The case $m=1$, $z=0$ was obtained in \cite{Hof1}:
\begin{equation}\label{eq.46}
\sum^\infty_{n=1}\frac{P_k(H_n,-H_n^{(2)},\ldots,(-1)^{k+1}H_n^{(k)})}{(n+2)^2}
=\sum^{k+2}_{j=2}\zeta(j)-(k+1).
\end{equation}
The following we present some examples with $(m,z)=(2,0)$ and $(2,-1/2)$.
For the case $(m,z)=(2,0)$, the general formula is 
\begin{eqnarray}\label{eq.47}
\lefteqn{\sum^\infty_{n=1}\frac{P_k(H_n^{(2)},-H_n^{(4)},\ldots,(-1)^{k+1}H_n^{(2k)})}{(n+2)^2}}\\ \nonumber
&=&(-1)^{k+1}(2k+1)+\zeta(\{2\}^{k+1})+2\sum^k_{j=1}(-1)^{k+1-j}(k+1-j)\zeta(\{2\}^j).
\end{eqnarray}
The following are the identities with $k=1,2,3$:
\begin{eqnarray*}
\sum^\infty_{n=1}\frac{H_n^{(2)}}{(n+2)^2}
&=& 3-2\zeta(2)+\zeta(2,2),\\
\sum^\infty_{n=1}\frac{(H_n^{(2)})^2-H_n^{(4)}}{2(n+2)^2}
&=& -5+4\zeta(2)-2\zeta(2,2)+\zeta(2,2,2),\\
\sum^\infty_{n=1}\frac{(H_n^{(2)})^3-3H_n^{(2)}H_n^{(4)}+2H_n^{(6)}}{6(n+2)^2}
&=& 7-6\zeta(2)+4\zeta(2,2)-2\zeta(2,2,2)+\zeta(2,2,2,2).
\end{eqnarray*}
The general formula of the case $(m,z)=(2,-1/2)$ is
\begin{eqnarray}\label{eq.48}
\lefteqn{\sum^\infty_{n=1}\frac{2^{(2k+2)}
P_k(O_n^{(2)},-O_n^{(4)},\ldots,(-1)^{k+1}O_n^{(2k)})}{(2n+3)^2}}\\ \nonumber
&=&(-1)^{k+1}(4k+4)+2^{2k+2}t(\{2\}^{k+1})
+2\sum^k_{j=1}(-1)^{k+1-j}(k+1-j)2^{2j}t(\{2\}^j).
\end{eqnarray}
The following are the identities with $k=1,2,3$:
\begin{eqnarray*}
2\sum^\infty_{n=1}\frac{O_n^{(2)}}{(2n+3)^2}
&=& 1-t(2)+2t(2,2),\\
2^3\sum^\infty_{n=1}\frac{(O_n^{(2)})^2-O_n^{(4)}}{(2n+3)^2}
&=& -3+2^2t(2)-2^3t(2,2)+2^4t(2,2,2),\\
2^4\sum^\infty_{n=1}\frac{(O_n^{(2)})^3-3O_n^{(2)}O_n^{(4)}+2O_n^{(6)}}{3(2n+3)^2}
&=& 2-3t(2)+2^3t(2,2)-2^4t(2,2,2)+2^5t(2,2,2,2).
\end{eqnarray*}

Let
\begin{equation}\label{eq.49}
E^{(m)}_{w,\ell}(z)=\sum_{\alpha_1+\cdots+\alpha_\ell=w}\zeta(m\alpha_1,\ldots,m\alpha_\ell;z).
\end{equation}
Hoffman \cite{Hof2}  investigated some properties of $E^{(2)}_{n,k}(0)$ and
Zhao \cite{Zhao} investigated some properties of $E^{(2)}_{n,k}(-1/2)$.
Chen et al. \cite{CCE} gave some formulas concerning $E^{(m)}_{n,k}(0)$.

From Theorem \ref{thm.32} we know that for $p\geq 2$ and $0\leq k\leq\ell$,
\begin{eqnarray*}
\eta_p(e_kh_{\ell-k}) &=& \sum^\ell_{j=k}{j\choose k}\eta_p(N_{\ell,j}) \\
&=& \sum^\ell_{j=k}{j\choose k}\sum_{\alpha_1+\cdots+\alpha_j=\ell}\eta_p(M_{(\alpha_1,\ldots,\alpha_j)}) \\
&=& \sum^\ell_{j=k}{j\choose k}\sum_{|\bm\alpha|=\ell}\left(
\zeta(m\alpha_1,\ldots,m\alpha_j+p;z)+\zeta(m\alpha_1,\ldots,m\alpha_j,p;z)\right).
\end{eqnarray*}
If we set $p=m$, then we can express $\eta_m(e_kh_{\ell-k})$ as a formula related $E^{(m)}_{w,\ell}(z)$:
\begin{proposition}
\begin{eqnarray*}
\lefteqn{\sum_{n=1}^\infty\frac{P_k(H_n^{(m)}(z),\ldots,(-1)^{k-1}H_n^{(km)}(z))
P_{\ell-k}(H_n^{(m)}(z),\ldots,H_n^{((\ell-k)m)}(z))}{(n+z)^m}}\\
&=& \sum^\ell_{j=k}{j-1\choose k-1}
\sum_{\alpha_1+\cdots+\alpha_j=\ell}\zeta(m\alpha_1,\cdots,m\alpha_j+m;z)
+\sum^{\ell+1}_{j=k+1}{j-1\choose k}E^{(m)}_{\ell+1,j}(z).
\end{eqnarray*}
\end{proposition}
\begin{proof}
\begin{eqnarray*}
\eta_m(e_kh_{\ell-k}) &=& 
\sum_{|\bm\alpha|=\ell}\left\{
\zeta(m\alpha_1,\ldots,m\alpha_k+m;z)
+\sum^\ell_{j=k+1}\left[
{j-1\choose k}\zeta(m\alpha_1,\ldots,m\alpha_{j-1},m;z)\right.\right.\\
&&\qquad\qquad\left.\left.
+{j\choose k}\zeta(m\alpha_1,\ldots,m\alpha_j+m;z)\right]\right\}
+{\ell\choose k}\zeta(\{m\}^{\ell+1};z).
\end{eqnarray*}
Since
$$
E^{(m)}_{\ell+1,j}(z)=\sum_{|\bm\alpha|=\ell}
\left[\zeta(m\alpha_1,\ldots,m\alpha_{j-1},m;z)+\zeta(m\alpha_1,\ldots,m\alpha_j+m;z)\right],
$$
and ${j\choose k}-{j-1\choose k}={j-1\choose k-1}$, we complet the proof.
\end{proof}

In general, we combine Lemma \ref{lma.22}, Theorem \ref{thm.34}, and Theorem \ref{thm.44}, then
we get our main theorem.
\begin{theorem}
For $u\in\QSym$ and $(s_1,\ldots,s_k)\in\mathbb N_0^k$ with 
$s_1+\cdots+s_k\geq 2$, the $H$-function $\eta_{s_1,\ldots,s_k}(u)$
can be explicitly expressed as a linear combination of multiple Hurwitz zeta functions
and some speical values of $H_n^{(m)}(z)$.
\end{theorem}
\section{Other Formulas and Applications}
Another interesting results are 
$$
\sum^\infty_{n=1}\frac{H_n^{(2)}}{{n+k\choose k}}
=\frac k{k-1}\left(\zeta(2)-H_{k-1}^{(2)}\right),
$$
where $k\geq 2$.
This result was first appeared in \cite[Conjecture 2.5]{Sof}.
In fact the above identity is just a formula obtained by evaluated $\eta_{0,1^k}(p_1)$ with $m=2$, $z=0$.

For nonnegative integers $a$, $c$ and using the mathematical induction
we get the following equations:
\begin{equation}\label{eq.51}
\eta_{0^a,1^{c+2}} = \frac1{(c+1)!}\sum^{c}_{r=0}(-1)^r{c\choose r}\eta_{0^{r+a},1^2}.
\end{equation}
Let $m=1=a$, $z=0$, we have the following identitiy \cite[Theorem 6]{Hof1}
\begin{example}\label{eg.51}
For a positive integer $k$ and a nonnegative integer $c$, we have
$$
\eta_{0,1^{c+2}}(e_k)=\sum^\infty_{n=1}
\frac{P_k(H_n,\ldots,(-1)^{k+1}H_n^{(k)})}{(n+1)(n+2)\cdots(n+c+2)}
=\frac1{(c+1)!}\cdot\frac1{(c+1)^{k+1}}.
$$
\end{example}
\begin{proof}
Using Eq.\,(\ref{eq.51}) and Proposition \ref{pro.35} we have
\begin{eqnarray*}
\lefteqn{(c+1)!\sum^\infty_{n=1}\frac{P_k(H_n,-H_n^{(2)},\ldots,(-1)^{k+1}H_n^{(k)})}
{(n+1)(n+2)\cdots(n+c+2)}}\\
&=& \sum^c_{r=0}{c\choose r}(-1)^r\left[\sum_{1\leq r_{k-1}\leq \cdots\leq r_0=r+1}
\frac{H_{r_{k-1}}}{\prod^{k-1}_{p=0}r_p}\right] \\
&=&\sum_{1\leq r_k\leq r_{k-1}\leq \cdots\leq r_0\leq c+1}
\frac{(-1)^{r_0+1}{c\choose r_0-1}}{r_0r_1\cdots r_k}.
\end{eqnarray*}
Since 
$$
{c\choose r_0-1}={c+1\choose r_0}-{c\choose r_0}
$$
and a known identity (ref. \cite{BDWLH})
\begin{equation}\label{eq.52}
H_c^{(k+1)}=\sum_{1\leq r_k\leq r_{k-1}\leq \cdots\leq r_0\leq c}
\frac{(-1)^{r_0+1}{c\choose r_0}}{r_0r_1\cdots r_k},
\end{equation}
The last equation becomes
$$
H_{c+1}^{(k+1)}-H_c^{(k+1)}=\frac1{(c+1)^{k+1}}.
$$
Thus we complete the proof.
\end{proof}
We can rewrite the identity in Example \ref{eg.51} as the following.
\begin{equation}\label{eq.53}
\sum^\infty_{n=1}\frac{P_\ell(H_n,-H_n^{(2)},\ldots,(-1)^{\ell+1}H_n^{(\ell)})}
{{n+k\choose k}}=\frac{k}{(k-1)^{\ell+1}},
\end{equation}
where $\ell\geq 1$, $k\geq 2$ are integers.

Similarly we set $m=1$, $a=z=0$ in Eq.\,(\ref{eq.51}), 
then the evaluation of the formula $\eta_{1^{c+2}}(e_k)$ gives the 
following formula \cite[Corollary 6]{Hof1}:
$$
(c+1)!\sum^\infty_{n=1}\frac{P_k(H_n,-H_n^{(2)},\ldots,(-1)^{k+1}H_n^{(k)})}
{n(n+1)\cdots(n+c+1)}=\zeta(k+1)-H_c^{(k+1)}.
$$
Again we can rewrite the above identity as the following.
\begin{equation}\label{eq.54}
\sum^\infty_{n=1}\frac{P_\ell(H_n,-H_n^{(2)},\ldots,(-1)^{\ell+1}H_n^{(\ell)})}
{n{n+k\choose k}}=\zeta(\ell+1)-H_{k-1}^{(\ell+1)},
\end{equation}
where $\ell\geq 1$, $k\geq 2$ are integers.

Now we apply our method to get the identity \cite{SS}
which we mention in the first section:
$$
\sum^\infty_{n=1}\frac{H_n}{n^4{n+5\choose 5}}
=3\zeta(5)-\zeta(3)\zeta(2)-\frac{137}{48}\zeta(4)+\frac{12019}{1800}\zeta(3)
-\frac{874853}{216000}\zeta(2)+\frac{131891}{172800}.
$$
Since 
\begin{eqnarray*}
\eta_{4,1^5} &=&
\frac1{120}\eta_4-\frac{137}{7200}\eta_3+\frac{12019}{432000}\eta_2
-\frac{12019}{432000}\eta_{1^2} \\
&&\quad-\frac{3799}{432000}\eta_{1^3}-\frac{1489}{216000}\eta_{1^4}
-\frac{61}{8000}\eta_{1^5}-\frac1{125}\eta_{1^6}.
\end{eqnarray*}
Let $m=1$ and $z=0$.
Applying Theorem \ref{thm.41} we have $\eta_p(e_1)=\zeta(p+1)+\zeta(1,p)$. Using the following 
known results:
$$
\zeta(1,4)=2\zeta(5)-\zeta(2)\zeta(3),\quad
\zeta(1,3)=\frac14\zeta(4),\quad
\zeta(1,2)=\zeta(3),
$$
we then get the desired result.

In the end of this section we consider the following identity \cite[Eq.\,(2.3)]{Sof2}:
\begin{eqnarray*}
\lefteqn{\sum^\infty_{n=1}\frac{H_n^{(m)}}{n{n+k\choose k}}
=\zeta(m+1)+\sum^k_{r=1}(-1)^{r+1}{k\choose r}}\\
&\times&\left[\sum^{r-1}_{j=1}\frac{(-1)^{m+1}}{j^m}H_j
+\sum^m_{\ell=2}(-1)^{m-\ell}H_{r-1}^{(m+1-\ell)}\zeta(\ell)\right].
\end{eqnarray*}
We apply our Main Theorem to evaluate 
this infinite series $\sum^\infty_{n=1}\frac{H_n^{(m)}}{n{n+k\choose k}}$.
Then we obtain a more efficient formula:
\begin{example}
For a pair of positive integers $k$ and $m$, we have
$$
\sum^\infty_{n=1}\frac{H_n^{(m)}}{n{n+k\choose k}}
=\zeta(m+1)+\sum^{k-1}_{r=1}(-1)^r{k-1\choose r}
\left[\frac{(-1)^{m-1}}{r^m}H_r+\sum^m_{\ell=2}\frac{(-1)^{m-\ell}}{r^{m+1-\ell}}\zeta(\ell)\right].
$$
\end{example}
\begin{proof}
Using Eq.\,(\ref{eq.51}) and set $z=0$, we have
\begin{eqnarray*}
\sum^\infty_{n=1}\frac{H_n^{(m)}}{n{n+k\choose k}}
&=& k!\eta_{1^{k+1}}(e_1)\ =\ \sum^{k-1}_{r=0}(-1)^r{k-1\choose r}\eta_{0^r,1^2}(e_1) \\
&=&\eta_{1,1}(e_1)+\sum^{k-1}_{r=1}(-1)^r{k-1\choose r}\eta_{0^r,1^2}(e_1).
\end{eqnarray*}
By Theorem \ref{thm.32}, $\eta_{1,1}(e_1)=\zeta(m+1)$. Applying Theorem \ref{thm.33} (1), 
for $r\geq 1$ we have
$$
\eta_{0^r,1^2}(e_1)=\frac{(-1)^{m-1}}{r^m}H_r+\sum^{m-2}_{\ell=0}\frac{(-1)^\ell}{r^{\ell+1}}\zeta(m-\ell).
$$
Substitute these two identities into the last formula, we obtain the desired result.
\end{proof}

\section*{Acknowledgment}
The author was funded by the Ministry of Science and Technology, Taiwan, R.O.C.,
under Grant MOST 107-2115-M-845-003.

\end{document}